\documentclass{article}[12pt]
\usepackage{amsmath}
\usepackage{amsthm}
\usepackage{amsfonts}
\usepackage{amssymb}
\usepackage{amstext}
\usepackage{amsopn}
\usepackage{latexsym}
\newtheorem{thm}{Theorem}[section]

\newtheorem{cor}{Corollary}[section]
\newtheorem{lemma}{Lemma}[section]

\begin{document}
\title{Convex Hulls of Dragon Curves\thanks{The work is supported by NSFC (No. 1250010257).}}

\author{Fan WEN \\ \\
Department of Mathematics\\
Jinan University\\
Guangzhou 510632, China\\
E-mail: wenfan@jnu.edu.cn}

\maketitle

\begin{abstract}
The fundamental geometry of self-similar sets becomes significantly more complex when the generating contractive maps include non-trivial rotational components. A well-known family exemplifying this complexity is that of the dragon curves in the plane. In this paper, we prove that  every dragon curve has a polygonal convex hull. Moreover, we completely characterize their convex hulls.

\medskip


\medskip

\end{abstract}

\section{Introduction}

Let $\{f_1,f_2,\dots, f_m\}$ be a family of contractive maps of $\mathbb{R}^d$, where $m\geq 2$ and $d\geq 1$ are integers. By Hutchinson's Theorem \cite{Hu}, there is a unique nonempty compact set $K\subset\mathbb{R}^d$ satisfying $$K=\bigcup_{j=1}^mf_j(K).$$ In this case, $\{f_1,f_2,\dots, f_m\}$ is called an iterated function system (IFS) on $\mathbb{R}^d$ and $K$ is called its attractor. We say that $K$ is self-similar if the IFS consists of similarity maps and self-affine if the IFS consists of affine maps.

The basic geometry of IFS attractors gives rise to many intriguing questions. For instance, determining when the convex hull of an IFS attractor is a polygon has been explored by multiple researchers.
Let $m\geq 2$ be an integer. Let $$f_j(x)=A_jx+d_j,\, x\in \mathbb{R}^d; \, j=1,2,\cdots m$$ where $A_j$'s are $d\times d$ contractive invertible matrices and $d_j$'s are points in $\mathbb{R}^d$. Let $K$ be the self-affine set of the IFS $\{f_1,f_2,\dots, f_m\}$. In the case $$A_1=\cdots= A_m=A,$$
Strichartz-Wang \cite{SW} proved that the convex hull of $K$ is a polygon if and only if there is an integer $s\geq 1$ such that $A^s$ is a scalar matrix.
In the general case, Kirat-Kocyigit \cite{KK} gave an algorithmic necessary and sufficient condition for the convex hull of any specific $K$ to be a polygon. Nevertheless, the method can not be applied to families with uncountable such sets.


Let $\mathbb{C}$ be the complex plane. For $z\in\mathbb{C}$ we denote by $|z|$, $\mbox{Re} z$, $\mbox{Im} z$, and $\bar{z}$ the absolute value, the real part, the imaginary part, and the conjugate of $z$ respectively. Denote by $\arg z$ the argument of $z$ in $[0,2\pi)$. For a subset $E$ of $\mathbb{C}$ we denote by $\mbox{co}(E)$ the convex hull of $E$.

Now we recall the definition of dragon curves. Let $\eta\in(0,{\pi}/{3})$.  Let
$$
a=\frac{e^{-i\eta}}{2\cos\eta}.
$$
Then $a+\bar{a}=1$, $2|a|\cos\eta=1$, $|a|\in(1/2,1)$, $\mbox{Re} a=1/2$, $\mbox{Im} a=-|a|\sin\eta$, $\arg a=2\pi-\eta$, $2|a|^2<1$ for $\eta\in(0,\pi/4)$, and $2|a|^2\geq1$ for $\eta\in[\pi/4,\pi/3)$.

Let $\{f_1, f_2\}$ be an IFS on $\mathbb{C}$, where
$$
f_1(z)=az \,\mbox{ and }\, f_2(z)=1-\bar{a}z.
$$
Then $f_1(0)=0$, $f_2(0)=1$, and $f_1(1)=f_2(1)=a$. The self-similar set of $\{f_1, f_2\}$ is called a dragon curve and is denoted by $K_\eta$. By \cite{H}, $K_\eta$ is path-connected.

The fundamental geometric properties of dragon curves are nontrivial. To date, the problem of characterizing when a dragon curve is an arc and when its IFS satisfies the open set condition (cf. \cite{Schief}) has only been partially solved; see \cite{AKW,A,Al,B,Fa,K,T}.
By using their algorithm, Kirat-Kocyigit \cite{KK} verified that the convex hull of the dragon curve $K_{\pi/4}$ is a polygon. In the present paper we prove that every dragon curve has a polygonal convex hull. Moreover, we completely characterize their convex hulls.

\begin{thm}\label{tm1}
Every dragon curve  has a polygonal convex hull.
\end{thm}

To prove Theorem \ref{tm1}, it suffices to show that, for each $\eta\in(0,\pi/3)$, the convex hull of $K_\eta$ coincides with the convex hull of a certain finite subset of $K_\eta$. Such a subset of $K_\eta$ is given as follows:

Let $z_0$ be the fixed point of $f_2^2\circ f_1^2$, where $f^{k}$ denotes the $k$-th iterate of $f$. Then $z_0=ca\in K_\eta$, where
\begin{equation}\label{ct}
c=\frac{1}{1-|a|^4}.
\end{equation}
For every integer $k\geq 0$ let
$$
z_k=f_1^{k}(z_0),\, w_k=f_2(z_k), \mbox{\, and \,} b_1=f_2(w_1).
$$
Since $z_0\in K_\eta$, one has $z_k,w_k, b_1\in K_\eta$. By computation, one has
\begin{equation}\label{a0}
z_k=ca^{k+1},\, w_k=1-c|a|^2a^k,\, b_1=a+c|a|^4,\,\mbox{ and }\,f_2(w_2)=z_0.
\end{equation}
Let
$$
V_k=\{b_1,z_0,z_1,\cdots,z_k, w_1\cdots, w_k\}
$$
and
$$V=\bigcup_{k=1}^\infty V_k.$$
Then $V$ is a countable subset of $K_\eta$. We shall show that, for every $\eta\in(0,\pi/3)$, there is an integer $k\geq 1$ such that
\begin{equation}\label{cb}
\mbox{co}(K_\eta)=\mbox{co}(V)=\mbox{co}(V_k),
\end{equation}
which gives Theorem \ref{tm1}.  A detailed proof will be given in Section 2.

\medskip

To characterize the convex hulls of dragon curves, for every $\eta\in(0,\pi/3)$ we hope to
find the smallest integer $k$ satisfying $\mbox{co}(K_\eta)=\mbox{co}(V_k)$ and prove that the points in $V_k$ are the vertices of $\mbox{co}(K_\eta)$.
We completely solve this problem by assigning every integer $k\geq 4$ a value $\eta_k$ in $(\pi/k, \pi/(k-1))$ and proving the following theorem.

\begin{thm}\label{tm2}
Let $k\geq 3$ be an integer and $\eta\in[\eta_{k+1},\eta_{k})$, where $\eta_3=\pi/3$. Then the vertices of $\mbox{co}(K_\eta)$ are $b_1,z_0,z_1,\cdots,z_k, w_1,\cdots,w_k$ in clockwise order.
\end{thm}

By Theorem \ref{tm2}, $\mbox{co}(K_\eta)$ is an octagon for every $\eta\in[\eta_4,\pi/3)$. The number of the vertices of $\mbox{co}(K_\eta)$ is a decreasing right-continuous step function of $\eta$ that takes all even values $\geq 8$ and tends to $\infty$ as $\eta$ tends to $0$.

For $u,v,u_1,u_2,\cdots, u_n\in \mathbb{C}$ we denote by $[u,v]$ the line segment from $u$ to $v$ and by $[u_1,u_2,\cdots, u_n]$ the broken line formed by the line segments $[u_1,u_2]$, $[u_2,u_3]$, $\cdots$, $[u_{n-1},u_n]$. We say that $[u_1,u_2,\cdots, u_n,u_1]$ is a convex polygonal broken line if it bounds a convex polygon of vertices $u_1,u_2,\cdots, u_n$.

Let $k\geq 3$ be an integer and $\eta\in[\eta_{k+1},\eta_{k})$. To prove Theorem \ref{tm2}, we first show that $[b_1,z_0,z_1,\cdots, z_k,w_1,w_2,\cdots, w_k,b_1]$ is a clockwise-oriented simple closed broken line and that its interior angle at each of the listed points is less than $\pi$. This shows that the broken line is convex polygonal, and hence the convex hull $\mbox{co}(V_k)$ is a polygon of vertices $$b_1,z_0,z_1,\cdots,z_k, w_1,w_2,\cdots,w_k$$ in clockwise order.
After that, we prove $$\{z_j: j> k\}\cup \{w_j: j> k\}\subset\mbox{co}(V_k),$$ which yields $\mbox{co}(V_k)=\mbox{co}(V)$. Then, since the equality $\mbox{co}(K_{\eta})=\mbox{co}(V)$ will be established in the proof of Theorem \ref{tm1}, we have $\mbox{co}(K_{\eta})=\mbox{co}(V_k)$.

The paper is organized as follows. In Section 2, we prove Theorem \ref{tm1} by showing that (\ref{cb}) holds for sufficiently large integer $k(\eta)$. In Section 3, we assign every integer $k\geq 4$ a value $\eta_k$ in $(\pi/k, \pi/(k-1))$. In Section 4, we prove that, for every integer $k\geq 3$ and every $\eta\in[\eta_{k+1},\eta_k)$, the convex hull $\mbox{co}(V_k)$ is a polygon of vertices $b_1,z_0,z_1,\cdots,z_k, w_1,w_2,\cdots,w_k$ in clockwise order, where $\eta_3=\pi/3$.
In Section 5, we verify $\mbox{co}(V_3)=\mbox{co}(V)$ for $\eta\in[\eta_4,\pi/3)$. In Section 6, we establish certain disk properties for dragon curves $K_\eta$ with $\eta\in(0,\eta_4)$. In Section 7, we prove $\mbox{co}(V_k)=\mbox{co}(V)$ for every integer $k\geq 4$ and every $\eta\in[\eta_{k+1},\eta_k)$ by combining the disk properties with a thorough analysis.

For the sake of convenience, we list some equalities that will be used later.
\begin{eqnarray*}
&& \cos3\eta=4\cos^3\eta-3\cos\eta=\frac{1-3|a|^2}{2|a|^3}. \\
&& \cos4\eta=8\cos^4\eta-8\cos^2\eta+1=\frac{1-4|a|^2+2|a|^4}{2|a|^4}.\\
&& \frac{\sin3\eta}{\sin\eta}=4\cos^2\eta-1=\frac{1-|a|^2}{|a|^2}. \\
&& \frac{\sin4\eta}{\sin\eta}=8\cos^3\eta-4\cos\eta=\frac{1-2|a|^2}{|a|^3}.\\
&& \frac{\sin5\eta}{\sin\eta}=16\cos^4\eta-12\cos^2\eta+1=\frac{1-3|a|^2+|a|^4}{|a|^4}.\\
&&\frac{\sin6\eta}{\sin\eta}=
\frac{1-4|a|^2+3|a|^4}{|a|^5}.\\
&&\frac{\sin7\eta}{\sin\eta}
=\frac{1-5|a|^2+6|a|^4-|a|^6}{|a|^6}.\\
&&\frac{\sin8\eta}{\sin\eta}=\frac{1-6|a|^2+10|a|^4-4|a|^6}{|a|^7}.
\end{eqnarray*}

\section{The Proof of Theorem \ref{tm1}}

Let $a$, $f_1$, $f_2$, $K_\eta$, $c$, $z_k$, $w_k$, $b_1$, $V_k$, and $V$ be defined as in Section 1. They all depend on the parameter $\eta$. For simplicity, we will seldom mark this dependence in our notation. Additionally, since their definitions and basic relations will be frequently applied, we shall not explicitly cite them on every occasion.

To prove Theorem \ref{tm1}, it suffices to show that, for every $\eta\in(0,\pi/3)$ there is an integer $k\geq 1$ such that $\mbox{co}(K_\eta)=\mbox{co}(V_k)=\mbox{co}(V)$.

\begin{lemma}\label{L2.1}
For every $\eta\in(0,\pi/3)$ there is a positive integer $k$ such that $$[0,1]\subset\mbox{co}(V)=\mbox{co}(V_k).$$
\end{lemma}

\begin{proof}
Let $\eta\in(0,\pi/3)$. By the definition, we see that $0$ lies in the interior of $\mbox{co}\{z_k:k\geq 1\}$. Since $f_2$ is a similarity map and $f_2(0)=1$, we see that $1$ lies in the interior of $\mbox{co}\{w_k:k\geq 1\}$. Thus $[0,1]$ lies in the interior of $\mbox{co}(V)$. Moreover, since
$$\lim_{k\to\infty}z_k=0 \,\mbox{ and }\, \lim_{k\to\infty}w_k=1,$$
there is an integer $k$ such that $\{z_j: j>k\}\cup\{w_j:j>k\}$ lies in the interior of $\mbox{co}(V)$. For such an integer $k$ we have $\mbox{co}(V)=\mbox{co}(V_k)$, as desired.
\end{proof}

\begin{lemma}\label{ylw} Let $\eta\in(0,\pi/3)$. For every integer $k\geq 0$ we have
$$f_{1}(w_{k+1})=(1-|a|^2)f_{1}(w_k)+|a|^2f_2(w_k).$$
\end{lemma}
\begin{proof}
By $a+\overline{a}=1$ we have
$$a^3-a^2+|a|^2a^2+|a|^4=-a^2\bar{a}+|a|^2a=0.$$
Then, for every integer $k\geq 0$ we have
\begin{eqnarray*}
&&(1-|a|^2)f_{1}(w_k)+|a|^2f_2(w_k)=(1-|a|^2)aw_k+|a|^2(1-\overline{a}w_k)\\
&=&(1-|a|^2)(a-c|a|^2a^{k+1})+|a|^2(a+c|a|^4a^{k-1})\\
&=& a-c|a|^2a^{k+1}+c|a|^4a^{k+1}+c|a|^6a^{k-1} \\
&=& a-c|a|^2a^{k+2}+c|a|^2a^{k+2}-c|a|^2a^{k+1}+c|a|^4a^{k+1}+c|a|^6a^{k-1} \\
&=& a(1-c|a|^2a^{k+1})+c|a|^2a^{k-1}(a^3-a^2+|a|^2a^2+|a|^4) \\
&=& aw_{k+1}=f_{1}(w_{k+1}),
\end{eqnarray*}
as desired.
\end{proof}

\begin{lemma}\label{1L2.1} Let $\eta\in(0,\pi/3)$. The following propositions hold.

1) $a\in\mbox{co}(V)$.

2) $f_1(b_1), f_2(b_1)\in\mbox{co}(V)$.

3) Either $w_0\in [0,1]$ or $w_0\in\mbox{co}\{0,z_2,z_3\}$.

4) $f_2(w_k)\in\mbox{co}(V)$ for all integers $k\geq 0$.

5) $f_1(w_k)\in\mbox{co}(V)$ for all integers $k\geq 0$.
\end{lemma}

\begin{proof}
1) Since $a=c^{-1}z_0$ and $c>1$, one has $a\in \mbox{co}(V)$ by $0,z_0\in \mbox{co}(V)$.

2) Since
$$f_{1}(b_1)=a(a+c|a|^4)=a^2+c|a|^4a=(1-|a|^4)z_1+|a|^4z_0,$$
we have $f_{1}(b_1)\in\mbox{co}(V)$.

Note that
\begin{eqnarray*}
f_2(b_1)&=&1-\bar{a}(a+c|a|^4)=1-|a|^2-c|a|^4\bar{a}\\
&=& (1-t)(1-|a|^2)+t(1-|a|^2-c|a|^4t^{-1}\bar{a}),
\end{eqnarray*}
where
$$t=\frac{|a|^4}{1-|a|^4+|a|^6}.$$
We easily verify $t\in(0,1)$. Thus, to prove $f_2(b_1)\in\mbox{co}(V)$, it suffices to show
$$1-|a|^2-c|a|^4t^{-1}\bar{a}\in\mbox{co}(V).$$
In fact, by the choice of $t$ and $a+\bar{a}=1=c(1-|a|^4)$, we have
\begin{eqnarray*}
&&1-|a|^2-c|a|^4t^{-1}\bar{a}=1-|a|^2-c(1-|a|^4+|a|^6)\bar{a}\\
&=&1-|a|^2-\bar{a}-c|a|^6\bar{a}=a-|a|^2-c|a|^6\bar{a}=a^2-c|a|^6\bar{a}\\
&=&a^2-c|a|^4a\bar{a}^2=a^2-c|a|^4a(1-a)^2\\
&=& (1+2c|a|^4)a^2-c|a|^4a-c|a|^4a^3\\
&=& c(1+|a|^4)a^2-c|a|^4a-c|a|^4a^3\\
&=& c(1-|a|^2+|a|^4)a^2+c|a|^2a^2-c|a|^4a-c|a|^4a^3\\
&=& (1-|a|^2+|a|^4)z_1+c|a|^2a(a-|a|^2)-c|a|^4a^3\\
&=& (1-|a|^2+|a|^4)z_1+c|a|^2a^3-c|a|^4a^3\\
&=& (1-|a|^2+|a|^4)z_1+(|a|^2-|a|^4)z_2\in\mbox{co}(V).
\end{eqnarray*}

3) By the definition, $w_0=1-c|a|^2$. One has $w_0\in[0,1]$ if $w_0\geq0$.

Next we suppose $w_0<0$. This assumption implies $$1-2|a|^2<1-|a|^2-|a|^4=c^{-1} w_0<0,$$
which yields $\cos\eta<1/\sqrt{2}$. Therefore the case $w_0<0$ is possible only if $\eta\in(\pi/4,\pi/3)$. For such $\eta$, it is clear that $z_2$ lies in the third quadrant and that $z_3$ lies in the second quadrant. We are going to show $w_0\in \mbox{co}\{0,z_2,z_3\}$. Let
$$
t=\frac{1}{|a|^2}-1.
$$
one has $t\in(0,1)$ and
\begin{eqnarray*}
(1-t)z_2+tz_3 &=& ca^3(1-t+ta)=ca^3(1-t\bar{a}) \\
&=& ca^2(a-t|a|^2)=ca^2(a-1+|a|^2)\\
&=& ca^2(-\bar{a}+|a|^2)=-ca^2\bar{a}^2=-c|a|^4.
\end{eqnarray*}
This shows that the line segment $[z_2,z_3]$ intersects the real axis at $-c|a|^4$. On the other hand, we have
$$w_0+c|a|^4=1-c|a|^2+c|a|^4=1-\frac{|a|^2-|a|^4}{1-|a|^4}=\frac{1}{1+|a|^2}>0,$$
and hence $w_0\in[-c|a|^4,0]$. The above facts yield $w_0\in \mbox{co}\{0,z_2,z_3\}$.

4) We are required to show $f_2(w_k)\in\mbox{co}(V)$ for all integers $k\geq 0$.

First, $f_2(w_0)\in\mbox{co}(V)$. In fact, by 3), either $w_0\in[0,1]$ or $w_0\in \mbox{co}\{0,z_2,z_3\}$. If $w_0\in[0,1]$, we have by $1, a\in\mbox{co}(V)$
$$f_2(w_0)=1-\bar{a}w_0=(1-w_0)+w_0a\in\mbox{co}(V).$$
If $w_0\in \mbox{co}\{0,z_2,z_3\}$, we have
$$f_2(w_0)\in\mbox{co}\{f_2(0),f_2(z_2),f_2(z_3)\}=\mbox{co}\{1,w_2,w_3\}\subset\mbox{co}(V).$$

Secondly, $f_2(w_1)\in\mbox{co}(V)$ since $f_2(w_1)=b_1$.

Thirdly, $f_2(w_k)\in\mbox{co}(V)$ for every integer $k\geq 2$ since
$$
f_2(w_k)=1-\bar{a}(1-c|a|^2a^k)=a+c|a|^4a^{k-1}=(1-|a|^4)z_0+|a|^4z_{k-2}.
$$

5) We are required to show $f_1(w_k)\in\mbox{co}(V)$ for all integers $k\geq 0$.

If $w_0\in[0,1]$, we have by $0,a\in\mbox{co}(V)$ $$f_1(w_0)=w_0a\in\mbox{co}(V).$$
If $w_0\in \mbox{co}\{0,z_2,z_3\}$, we have
$$f_1(w_0)\in\mbox{co}\{f_1(0),f_1(z_2),f_1(z_3)\}=\mbox{co}\{0,z_3,z_4\}\subset\mbox{co}(V).$$
This proves $f_1(w_0)\in\mbox{co}(V)$. Now, by using Lemma \ref{ylw} and $f_2(w_k)\in\mbox{co}(V)$, we inductively get $f_1(w_k)\in\mbox{co}(V)$ for all integers $k\geq 0$.
\end{proof}

\begin{lemma}\label{cons}
Let $\eta\in(0,\pi/3)$. Then $f_1(V)\cup f_2(V)\subset\mbox{co}(V)$.
\end{lemma}

\begin{proof}
By the definitions, we have
$$f_1(V)=\{f_{1}(b_1)\}\cup\{z_k: k\geq 1\}\cup \{f_{1}(w_k):k\geq 1\}$$
and $$f_2(V)=\{f_{2}(b_1)\}\cup\{w_k: k\geq0\}\cup \{f_2(w_k): k\geq 1\}.$$
Then we get the desired result by Lemma \ref{1L2.1}.
\end{proof}

\noindent{\bf The Proof of Theorem \ref{tm1}.} Let $\eta\in(0,\pi/3)$. Since $V\subset K_\eta$, one has $\mbox{co}(V)\subseteq \mbox{co}(K_\eta)$. On the other hand, by Lemma \ref{L2.1} and Lemma \ref{cons}, $\mbox{co}(V)$ is a closed polygon with $f_1(\mbox{co}(V))\cup f_2(\mbox{co}(V))\subset\mbox{co}(V)$. Since $K_\eta$ is the attractor of the IFS $\{f_1,f_2\}$, we then get $K_\eta\subseteq\mbox{co}(V)$, which yields $\mbox{co}(K_\eta)\subseteq\mbox{co}(V)$. Thus $\mbox{co}(K_\eta)=\mbox{co}(V)$, so $\mbox{co}(K_\eta)$ is a polygon.

\section{The Desired Sequence $\{\eta_k\}_{k=4}^\infty$}

We define a sequence of functions on $(0,{\pi}/{3})$ by
$$\Phi_{k}(\eta)=(1-|a|^4)\sin (k-1)\eta-|a|^3\sin(k-2)\eta+|a|^{k}\sin\eta,$$
where $|a|=(2\cos\eta)^{-1}$. Since $$|a|^2\sin (k-1)\eta=|a|^3\sin(k-2)\eta+|a|^3\sin k\eta,$$
the function $\Phi_{k}$ can be equivalently defined by
$$\Phi_{k}(\eta)= (1-|a|^2-|a|^4)\sin (k-1)\eta+|a|^3\sin k\eta+|a|^{k}\sin\eta.$$
In this section, we prove that, for every integer $k\geq 4$, the function $\Phi_{k}(\eta)$ has a unique root (the desired $\eta_k$) in the interval $(\pi/k, \pi/(k-1))$. The following results motivate the definition of $\Phi_{k}(\eta)$.

\begin{lemma}\label{Im}
Let $k\geq 1$ be an integer. For every $\eta\in(0,\pi/3)$ we have
$$
\Phi_{k}(\eta)\asymp\mbox{Im}\frac{w_1-z_{k}}{z_{k-1}-z_k}.
$$
Hereafter we write $u\asymp v$ if $u=cv$ for some $c>0$.
\end{lemma}
\begin{proof}
Let $k\geq1$ be an integer and $\eta\in(0,\pi/3)$.
By the definitions of $a$, $c$, $z_k$, and $w_k$ in Section 1, we have
$$z_{k-1}-z_k\asymp a^k-a^{k+1}=a^k(1-a)=a^k\bar{a}\asymp a^{k-1}$$
and
\begin{eqnarray*}
\mbox{Im}\frac{w_1-z_{k}}{z_{k-1}-z_k}&\asymp& \mbox{Im}((\bar{z}_{k-1}-\bar{z}_k)(w_1-z_k))=\mbox{Im}((\bar{z}_{k-1}-\bar{z}_k)(w_1-z_{k-1}))\\
&\asymp&\mbox{Im}(\bar{a}^{k-1}(1-c|a|^2a-ca^k))\\
&\asymp & \mbox{Im}(\bar{a}^{k-1}-c|a|^4\bar{a}^{k-2}-c|a|^{2k-2}a))\\
&\asymp&\sin (k-1)\eta-c|a|^3\sin(k-2)\eta+c|a|^k\sin\eta\\
&\asymp&(1-|a|^4)\sin (k-1)\eta-|a|^3\sin(k-2)\eta+|a|^{k}\sin\eta.
\end{eqnarray*}
This completes the proof. \end{proof}

For $u,v,w\in \mathbb{C}$ let $$\angle uvw=\arg \frac{w-v}{u-v}.$$ Note that $\arg z$ denotes the argument of $z$ in $[0,2\pi)$. By the definition, $\angle uvw$ is the anticlockwise angle from  $u-v$ to $w-v$ and we have that
\begin{equation}\label{ag}
\angle uvw\in(0,\pi)\Leftrightarrow \mbox{Im}\frac{w-v}{u-v}>0.
\end{equation}

\begin{cor}\label{sp}
For every $\eta\in(0,\pi/3)$ we have $\angle z_2z_3w_1\in(0,\pi)$.
\end{cor}

\begin{proof} By the definition, $\Phi_3(\eta)=(1-|a|^4)\sin 2\eta>0$ for every $\eta\in(0,\pi/3)$. Then,
by Lemma \ref{Im} and (\ref{ag}), we have $\angle z_2z_3w_1\in(0,\pi)$.
\end{proof}

\begin{lemma}\label{L1.2}
For every integer $k\geq 4$ the function $\Phi_{k}(\eta)$ has a unique root in the interval $(\pi/k, \pi/(k-1))$.
\end{lemma}
\begin{proof} We first show that $\Phi_4(\eta)$ has a unique root in $(\pi/4, \pi/3)$. In fact,
\begin{eqnarray*}
\Phi_4(\eta)&=& (1-|a|^4)\sin 3\eta-|a|^3\sin 2\eta+|a|^4\sin\eta \\
&\asymp& (1-|a|^4)\frac{\sin 3\eta}{\sin\eta}-|a|^2+|a|^4\\
&\asymp& (1-|a|^4)\frac{1-|a|^2}{|a|^2}-|a|^2+|a|^4\\
&\asymp&\frac{1-|a|^4}{|a|^2}-|a|^2\asymp 1-2|a|^4\asymp8\cos^4\eta-1.
\end{eqnarray*}
Therefore $\Phi_4(\eta)=0$ if and only if $8\cos^4\eta=1$, which has a unique root in $(\pi/4, \pi/3)$. Denote this root by $\eta_4$.
Then one has
\begin{equation}\label{eta4}
2|a|^4<1  \mbox{ for }  \eta\in(0,\eta_4)
\end{equation}
and
\begin{equation}\label{t4}
2|a|^4\geq1 \mbox{ for } \eta\in[\eta_4,\pi/3).
\end{equation}

Now we show that $\Phi_{k}(\eta)$, $k\geq 5$, has a unique root in $(\pi/k, \pi/(k-1))$. By the definition, we easily verify 
$$\Phi_{k}(\frac{\pi}{k})>0\,\mbox{ and }\,\Phi_{k}(\frac{\pi}{k-1})<0.$$
Next we show that the derivative $\Phi'_{k}(\eta)<0$ for all $\eta\in(\pi/k,\pi/(k-1))$.

By $2|a|\cos\eta=1$, one has $d|a|/d\eta=2|a|^2\sin\eta$. Thus
\begin{eqnarray*}
\Phi'_{k}(\eta)&=&-8|a|^5\sin\eta\sin (k-1)\eta +(k-1)(1-|a|^4)\cos (k-1)\eta\\ &&
-6|a|^4\sin\eta\sin(k-2)\eta -(k-2)|a|^3\cos(k-2)\eta\\ &&+2k|a|^{k+1}\sin^2\eta+|a|^{k}\cos\eta.
\end{eqnarray*}
For $\eta\in(\pi/k,\pi/(k-1))$ one has $\sin (k-1)\eta>0$, $\sin(k-2)\eta>0$,
$$\cos (k-1)\eta<\cos\frac{(k-1)\pi}{k}=-\cos\frac{\pi}{k}<-\cos\eta,$$
and
$$\cos (k-2)\eta>\cos\frac{(k-2)\pi}{k-1}=-\cos\frac{\pi}{k-1}>-\cos\eta.$$
Moreover, for $k\geq 5$ and $\eta\in(\pi/k,\pi/(k-1))$ one has $k\sin^2\eta+\cos^2\eta<3$. In fact, if $k=5$, one has $$k\sin^2\eta+\cos^2\eta=1+4\sin^2\eta<1+4\sin^2\frac{\pi}{4}=3.$$ If $k\geq 6$, one has $$k\sin^2\eta+\cos^2\eta=1+(k-1)\sin^2\eta\leq1+(k-1)(\frac{\pi}{k-1})^2\leq1+\frac{\pi^2}{5}<3.$$
In addition, one has $2|a|^2<1$. These facts yield
\begin{eqnarray*}
\Phi'_{k}(\eta)&<&(-(k-1)(1-|a|^4)+(k-2)|a|^3)\cos\eta+2|a|^6(k\sin^2\eta+\cos^2\eta)\\
&<&  (-(k-1)(1-|a|^4)+(k-2)|a|^3)\cos\eta+6|a|^6\\
&\asymp&-(k-1)(1-|a|^4)+(k-2)|a|^3+12|a|^7\\
&<& -\frac{3(k-1)}{4}+\frac{k-2}{2\sqrt{2}}+\frac{3}{2\sqrt{2}}\\
&\leq&-\frac{(3-\sqrt{2})(k-1)}{4}+\frac{1}{\sqrt{2}}\\
&<&-3+\sqrt{2}+\frac{1}{\sqrt{2}}<0.
\end{eqnarray*}
It then follows that $\Phi_k(\eta)$ is strictly decreasing on $(\pi/k, \pi/(k-1))$. This proves that $\Phi_k(\eta)$  has a unique root in this interval.
\end{proof}

Let $\eta_3=\pi/3$. For every integer $k\geq 4$ let $\eta_k$ be the unique root of $\Phi_k(\eta)$ in $(\pi/k, \pi/(k-1))$. Then we have the decomposition
\begin{equation}\label{par}
(0,\frac{\pi}{3})=\bigcup_{k=3}^\infty[\eta_{k+1},\eta_k).
\end{equation}

\begin{cor}\label{sp1}
Let $k\geq 4$ be an integer and $\eta\in (0, \eta_k)$. Then we have $\angle z_{k-1}z_kw_1\in (0,\pi)$.
\end{cor}
\begin{proof}
In the case $\eta\in[\pi/k,\eta_k)$, one has $\Phi_k(\eta)>0$ by the proof of Lemma \ref{L1.2}. In the case $\eta\in(0,{\pi}/{k})$, one has
$$1-|a|^2-|a|^4>0,\ \sin (k-1)\eta>0,\, \mbox{ and }\,\sin k\eta>0,$$
which implies
$$
\Phi_{k}(\eta)=(1-|a|^2-|a|^4)\sin (k-1)\eta+|a|^3\sin k\eta+ |a|^k\sin\eta>0.$$
This proves $\Phi_k(\eta)>0$ for $\eta\in (0, \eta_k)$. It then follows from Lemma \ref{Im} and (\ref{ag}) that $\angle z_{k-1}z_{k}w_1\in(0,\pi)$.
\end{proof}

\section{The Convex Hull of $V_k$}

Let $k\geq 3$ be an integer and $\eta\in[\eta_{k+1},\eta_k)$. In this section, we verify that the broken line $[b_1,z_0,z_1,\cdots, z_k,w_1,w_2,\cdots, w_k,b_1]$ is a simple closed path with a clockwise orientation, whose interior angle at each of the listed points is less than $\pi$. This shows that $\mbox{co}(V_k)$ is a polygon of vertices $$b_1,z_0,z_1,\cdots,z_k, w_1,w_2,\cdots,w_k$$ in clockwise order.
For the purpose the following general result is needed.
\begin{lemma}\label{q0}
Let $\eta\in(0,\pi/3)$. The following propositions hold.

1) $\mbox{Re} z_3<\mbox{Re} b_1$ and $\mbox{Re} z_3<\mbox{Re} w_1<\mbox{Re} w_2$.

2) $\mbox{Im} b_1<\mbox{Im} z_3<\mbox{Im} w_1=\mbox{Im} w_2$.

3) $\angle z_2z_3w_1,\,\angle z_3w_1w_2,\, \angle w_3b_1z_0,\, \angle z_3w_1b_1,\, \angle w_1b_1z_0\in(0,\pi)$.

4) $\angle b_1z_0z_1=\angle z_jz_{j+1}z_{j+2}=\angle w_jw_{j+1}w_{j+2}=\pi-\eta$ for $j\geq 0$.

5) $[w_1,b_1]$ intersects the real axis at $(1+c^2|a|^6)/(1+c|a|^2)$.

6) $[w_1,z_1]$ intersects the real axis at $(1-c|a|^4)/(1+|a|^2)$.
\end{lemma}
\begin{proof}

1) Since $b_1=a+c|a|^4$ and $z_k=ca^{k+1}$, we have
$$\mbox{Re}b_1-\mbox{Re}z_k=|a|\cos\eta+c|a|^4 -c|a|^4\cos4\eta>0,$$
Thus $\mbox{Re} z_3<\mbox{Re} b_1$.

Since $w_j=1-c|a|^2a^j$, we have
\begin{equation}\label{ri}
w_2-w_1=c|a|^2a(1-a)=c|a|^4,
\end{equation}
which gives $\mbox{Re} w_2>\mbox{Re} w_1$.

Finally, since
\begin{eqnarray*}
\mbox{Re}w_1-\mbox{Re} z_3&=&1-c|a|^3\cos\eta-c|a|^4\cos4\eta\\
&=&1-c|a|^3\cos\eta-\frac{c}{2}(1-4|a|^2+2|a|^4)\\
&=&1-\frac{c|a|^2}{2}-\frac{c}{2}+2c|a|^2-c|a|^4\\
&\asymp&2-c+3c|a|^2-2c|a|^4\\
&\asymp&1+3|a|^2-4|a|^4>0,
\end{eqnarray*}
we get $\mbox{Re} z_3<\mbox{Re} w_1$.

\medskip

2) By (\ref{ri}), we have $\mbox{Im} w_2=\mbox{Im} w_1$.

Since
\begin{eqnarray*}
\mbox{Im} w_1-\mbox{Im} z_3&=&c|a|^3\sin\eta+c|a|^4\sin4\eta\asymp\sin2\eta+\sin4\eta\\
&\asymp& 1+2\cos2\eta=4\cos^2\eta-1\asymp1-|a|^2>0,
\end{eqnarray*}
we have $\mbox{Im} z_3<\mbox{Im} w_1$.

Finally, since
\begin{eqnarray*}
\mbox{Im}b_1-\mbox{Im}z_3&=& -|a|\sin\eta+c|a|^4\sin4\eta \\
&\asymp& -1+|a|^4+|a|^3\frac{\sin4\eta}{\sin\eta}\\
&=& |a|^4-2|a|^2<0,
\end{eqnarray*}
we get $\mbox{Im}b_1<\mbox{Im}z_3$.

\medskip

3) By Corollary \ref{sp}, we have $\angle z_2z_3w_1\in(0,\pi)$.

By 1) and 2), we have $\mbox{Re} w_1<\mbox{Re} w_2$ and $\mbox{Im} z_3<\mbox{Im} w_1=\mbox{Im} w_2$, and hence $\angle z_3w_1w_2\in(0,\pi)$.

Since $f_2$ preserves angles, we get $\angle w_3b_1z_0=\angle z_3w_1w_2\in(0,\pi)$.

By 1) and 2), we have $\mbox{Re}z_3<\mbox{Re}b_1$, $\mbox{Re}z_3<\mbox{Re}w_1$ and $\mbox{Im}b_1<\mbox{Im}z_3<\mbox{Im}w_1$, and hence $\angle z_3w_1b_1\in(0,\pi)$.

Since
\begin{eqnarray*}
&&\mbox{Im}\frac{z_0-b_1}{w_1-b_1}\asymp\mbox{Im}((\bar{w}_1-\bar{b}_1)(z_0-b_1))=\mbox{Im}(z_0(\bar{w}_1-\bar{b}_1)-\bar{w}_1b_1)\\
&=& \mbox{Im}(ca(1-c|a|^2\bar{a}-\bar{a}-c|a|^4)-(1-c|a|^2\bar{a})(a+c|a|^4))\\
&=& \mbox{Im}(ca-c^2|a|^4a-a+c^2|a|^6\bar{a})\\
&=& -(c-1)|a|\sin\eta +c^2|a|^5\sin\eta+c^2|a|^7\sin\eta\\
&\asymp& 1-c +c^2|a|^4+c^2|a|^6=-c|a|^4+c^2|a|^4+c^2|a|^6\\
&\asymp& c-1+c|a|^2=c|a|^4+c|a|^2>0,
\end{eqnarray*}
we have $\angle w_1b_1z_0\in(0,\pi)$ by (\ref{ag}).

\medskip

4) By $a+\bar{a}=1$, one has $$
\frac{z_1-z_0}{b_1-z_0}=\frac{c(a^2-a)}{a+c|a|^4-ca}=\frac{a^2-a}{(1-|a|^4)a+|a|^4-a}=-\frac{a}{|a|^4}$$ and $$\frac{z_{j+2}-z_{j+1}}{z_j-z_{j+1}}=\frac{a^2-a}{1-a}
=-a.
$$
Therefore
$$
\angle b_1z_0z_1=\angle z_jz_{j+1}z_{j+2}=\pi-\eta.
$$
Since $f_2$ preserves angles, we have $\angle w_jw_{j+1}w_{j+2}=\angle z_jz_{j+1}z_{j+2}=\pi-\eta$.

\medskip

5) For $t\in[0,1]$ we have
\begin{eqnarray*}
\mbox{Im}(tw_1+(1-t)b_1) &=&\mbox{Im}(-tc|a|^2a+(1-t)a) \\
&=& tc|a|^3\sin\eta-(1-t)|a|\sin\eta\\
&\asymp& tc|a|^2+t-1.
\end{eqnarray*}
Thus, $\mbox{Im}(tw_1+(1-t)b_1)=0$ if and only if $t=1/(1+c|a|^2)$. For this value of $t$, we have by (\ref{eta4})
$$
tw_1+(1-t)b_1 = \frac{1-c|a|^2a+c|a|^2(a+c|a|4)}{1+c|a|^2} =\frac{1+c^2|a|^6}{1+c|a|^2}.$$
Therefore $[w_1,b_1]$ intersects the real axis at $(1+c^2|a|^6)/(1+c|a|^2)$.

6) By $2|a|\cos\eta=1$, we have for $t\in[0,1]$
\begin{eqnarray*}
\mbox{Im}(tw_1+(1-t)z_1) &=&\mbox{Im}(-tc|a|^2a+(1-t)ca^2) \\
&=& tc|a|^3\sin\eta-(1-t)c|a|^2\sin2\eta\\
&\asymp& t|a|^2+t-1.
\end{eqnarray*}
Thus, $\mbox{Im}(tw_1+(1-t)z_1)=0$ if and only if $t=1/(1+|a|^2)$. For this value of $t$, we have by $a+\bar{a}=1$ and (\ref{eta4})
$$tw_1+(1-t)z_1=\frac{1-c|a|^2a+c|a|^2a^2}{1+|a|^2}=\frac{1-c|a|^4}{1+|a|^2}.$$
Therefore $[w_1,z_1]$ intersects the real axis at $(1-c|a|^4)/(1+|a|^2)$.
\end{proof}

\begin{lemma}\label{q1}
For every $\eta\in(\pi/4,\pi/3)$ we have  $\angle w_10z_3\in(0,\pi)$.
\end{lemma}
\begin{proof}
For $\eta\in(\pi/4,\pi/3)$ we have $\mbox{Im} z_3=-c|a|^4\sin4\eta>0$. It then follows from
$\mbox{Im}w_1=c|a|^3\sin\eta>0$ and $\mbox{Re} z_3<\mbox{Re} w_1$ that  $\angle w_10z_3\in(0,\pi)$.
\end{proof}

\begin{lemma}\label{01}
\label{q2} Let $k\geq 4$ be an integer and $\eta\in[\eta_{k+1},\eta_k)$. We have the following results.

1) $\mbox{Re} z_k<0<\mbox{Im} z_k$.

2) $\mbox{Re} z_k<\mbox{Re} w_1$ and $\mbox{Im} z_k<\mbox{Im} w_1$.

3) $\angle w_10z_k,\, \angle z_{k-1}z_kw_1,\, \angle z_kw_1w_2,\, \angle w_kb_1z_0,\, \angle z_kw_1b_1\in(0,\pi)$.
\end{lemma}
\begin{proof} 1) By the assumption, one has $\pi<(k+1)\eta<\pi+2\eta<2\pi$. Thus
$$\mbox{Im} z_k=-c|a|^{k+1}\sin(k+1)\eta>0.$$

Next we prove $\mbox{Re} z_k<0$. Clearly,
$$\mbox{Re} z_k=c|a|^{k+1}\cos(k+1)\eta.$$

If $k=4$, one has $\pi<5\eta<5\eta_4<2\pi$, so $\cos5\eta<\cos5\eta_4$. Since $8\cos^4\eta_4=1$ by the definition of $\eta_4$ in the proof of Lemma \ref{L1.2}, one has
\begin{eqnarray*}
&&\cos5\eta_4 =16\cos^5\eta_4-20\cos^3\eta_4+5\cos\eta_4 \\
&=&7\cos\eta_4-\frac{5}{2\cos\eta_4}\asymp 14\cos^2\eta_4-5=\frac{14}{\sqrt{8}}-5<0.
\end{eqnarray*}
Thus $\cos5\eta<0$ and we get $\mbox{Re} z_4<0$.

If $k\geq 5$, one has $\pi<(k+1)\eta<3\pi/2$, which implies $\cos(k+1)\eta<0$, and hence $\mbox{Re} z_k<0$.

2) By $\pi<(k+1)\eta<\pi+2\eta<2\pi$, one has $\cos(k+1)\eta<-\cos2\eta$. Thus
\begin{eqnarray*}
\mbox{Re} w_1-\mbox{Re} z_k &=& 1-c|a|^3\cos\eta-c|a|^{k+1}\cos(k+1)\eta\\
&>& 1-c|a|^3\cos\eta+c|a|^{k+1}\cos2\eta\\
&=& 1-c|a|^3\cos\eta+2c|a|^{k+1}\cos^2\eta-c|a|^{k+1}\\
&=& 1-\frac{c|a|^2}{2}+\frac{c|a|^{k-1}}{2}-c|a|^{k+1}\\
&\asymp& 1-|a|^4-\frac{|a|^2}{2}+\frac{|a|^{k-1}}{2}-|a|^{k+1}\\
&\asymp& 2-|a|^2-2|a|^4+|a|^{k-1}(1-2|a|^2)
\end{eqnarray*}
If $2|a|^2\leq 1$, one has
$$2-|a|^2-2|a|^4+|a|^{k-1}(1-2|a|^2)\geq 3|a|^2-2|a|^4>0.$$
If $2|a|^2>1$, one has by (\ref{eta4})
\begin{eqnarray*}
&&2-|a|^2-2|a|^4+|a|^{k-1}(1-2|a|^2)\\
&>& 2-|a|^2-2|a|^4+|a|^3(1-2|a|^2) \\
&=&  2-|a|^2+|a|^3-2|a|^4-2|a|^5\\
&>&  1-|a|-|a|^2+|a|^3\\
&=&(1-|a|)(1-|a|^2)>0.
\end{eqnarray*}
This proves $\mbox{Re} w_1>\mbox{Re} z_k$.

Next we prove $\mbox{Im} z_k<\mbox{Im} w_1$. If $k=4$, one has
\begin{eqnarray*}
\mbox{Im} w_1-\mbox{Im} z_4&\asymp&1+|a|^{2}\frac{\sin5\eta}{\sin\eta}=1+\frac{1-3|a|^2+|a|^4}{|a|^2}\\&\asymp& 1-2|a|^2+|a|^4>0.
\end{eqnarray*}
If $k\geq 5$, one has $\pi<(k+1)\eta<\pi+2\eta<3\pi/2$, so $\sin(k+1)\eta>-\sin2\eta$, by which we have
\begin{eqnarray*}
\mbox{Im} w_1-\mbox{Im} z_k&\asymp&\sin\eta+|a|^{k-2}\sin(k+1)\eta\\
&>& \sin\eta-|a|^{k-2}\sin2\eta \\
&\asymp& 1-|a|^{k-3}>0.
\end{eqnarray*}

3) As we proved in 1) and 2), we have $0<\mbox{Im} z_k<\mbox{Im} w_1$ and $\mbox{Re} z_k<\mbox{Re} w_1$, and hence $\angle w_10z_k\in(0,\pi)$.

By Corollary \ref{sp1}, we have $\angle z_{k-1}z_kw_1\in(0,\pi)$.

By Lemma \ref{q0}, we have $\mbox{Re} w_1<\mbox{Re} w_2$ and $\mbox{Im} w_1=\mbox{Im} w_2$. It then follows from $\mbox{Im} z_k<\mbox{Im} w_1$
that $\angle z_kw_1w_2\in(0,\pi)$.

Since $f_2$ preserves angles, we have $\angle w_kb_1z_0=\angle z_kw_1w_2\in(0,\pi)$.

Since
$$\mbox{Re}b_1-\mbox{Re}z_k=|a|\cos\eta+c|a|^4 -c|a|^{k+1}\cos(k+1)\eta>0,$$
we have $\mbox{Re}z_k<\mbox{Re}b_1$. In addition, we have $\mbox{Im}b_1=-|a|\sin\eta<0$. To sum up, we have
$\mbox{Re}z_k<\mbox{Re}b_1$, $\mbox{Re} z_k<\mbox{Re} w_1$, and $\mbox{Im}b_1<\mbox{Im}z_k<\mbox{Im} w_1$, which yields  $\angle z_kw_1b_1\in(0,\pi)$.
\end{proof}

\begin{lemma}\label{q3}
Let $k\geq 3$ be an integer and $\eta\in[\eta_{k+1},\eta_k)$. Then $\mbox{co}(V_k)$ coincides with the polygon of vertices $b_1,z_0,z_1,\cdots, z_k,w_1,w_2,\cdots, w_k$ in clockwise order and contains $[0,1]$.
\end{lemma}

\begin{proof} We first study the broken line $[b_1,z_0,z_1,\cdots, z_k,w_1,b_1]$. By Lemmas \ref{q1} and \ref{q2}, we have $\angle w_10z_k\in(0,\pi)$. In addition, we have $\mbox{Im}w_1=c|a|^3\sin\eta>0$. Therefore $\arg z_k>\arg w_1$. On the other hand, since $\arg a=2\pi-\eta$, $z_0=ca$, and $b_1=a+c|a|^4$, we have $2\pi>\arg b_1>\arg z_0$.
By combining these two inequalities for the arguments, we easily see that
$$2\pi>\arg b_1>\arg z_0>\arg z_1>\cdots >\arg z_{k-1}>\arg z_k>\arg w_1>0.$$
Now, since $[w_1,b_1]$ intersects the positive real axis by Lemma \ref{q0}, we conclude that $[b_1,z_0,z_1,\cdots, z_k,w_1,b_1]$ is a clockwise-oriented simple closed broken line enclosing $0$. By Lemmas \ref{q0} and \ref{01}, we have
$$\angle b_1z_0z_1, \angle z_jz_{j+1}z_{j+2},\angle z_{k-1}z_kw_1, \angle z_kw_1b_1,\angle w_1b_1z_0\in(0,\pi).$$
That is to say, the interior angle of the broken line $[b_1,z_0,z_1,\cdots, z_k,w_1,b_1]$ at each of the listed points is less than $\pi$. Thus it is a clockwise-oriented convex polygonal broken line.

The last conclusion implies that $[z_1,z_2,\cdots, z_k,w_1,z_1]$ is a clockwise-oriented convex polygonal broken line. Since $f_2$ preserves angles and $$f_2[z_1,\cdots, z_k,w_1,z_1]=[w_1,\cdots, w_k,b_1,w_1],$$ we see that $[w_1,\cdots, w_k,b_1,w_1]$ is also such a polygonal broken line.

The broken lines $[b_1,z_0,z_1,\cdots, z_k,w_1,b_1]$ and $[w_1,w_2,\cdots, w_k,b_1,w_1]$ share a common side $w_1b_1$.
Since they are clockwise-oriented and convex, it follows that $[b_1,z_0,z_1,\cdots, z_k,w_1,w_2,\cdots, w_k,b_1]$ is a clockwise-oriented simple closed broken line. By Lemmas \ref{q0} and \ref{01}, we have
$$\angle z_kw_1w_2, \angle w_kb_1z_0\in(0,\pi),$$
and hence $[b_1,z_0,z_1,\cdots, z_k,w_1,w_2,\cdots, w_k,b_1]$ is a clockwise-oriented convex polygonal broken line. This proves that $\mbox{co}(V_k)$ coincides with the polygon of vertices $b_1,z_0,z_1,\cdots, z_k,w_1,w_2,\cdots, w_k$ in clockwise order.

Next we prove $[0,1]\subset \mbox{co}(V_k)$.
If $k=3$ and $\eta\in [\eta_4,\pi/3)$, one has $2|a|^4\geq 1$ by (\ref{t4}), and hence $c|a|^4\geq 1$, which gives $$\frac{1+c^2|a|^6}{1+c|a|^2}\geq 1.$$ Since $[w_1,b_1]$ intersects the positive real axis at $(1+c^2|a|^6)/(1+c|a|^2)$, we have that $[b_1,z_0,z_1,\cdots, z_k,w_1,b_1]$ encloses $0$ and $1$, and hence $[0,1]\subset \mbox{co}(V_k)$.

If $k\geq4$ and $\eta\in [\eta_{k+1},\eta_k)$, one has $2|a|^4<1$ by (\ref{eta4}), and hence $c|a|^4<1$, which yields
$$\frac{1-c|a|^4}{1+|a|^2}\in(0,1).$$
Since $[w_1,z_1]$ intersects the real axis at $(1-c|a|^4)/(1+|a|^2)$, we see that the broken line $[z_1,z_2,\cdots, z_k,w_1,z_1]$ encloses $0$. By the action of $f_2$, we conclude that $[w_1,w_2,\cdots, w_k,b_1,w_1]$ encloses $1$. Thus $[0,1]\subset \mbox{co}(V_k)$.
\end{proof}

\section{$\mbox{co}(V)=\mbox{co}(V_3)$ for $\eta\in[\eta_4,\pi/3)$}

In this section we prove that the equality $\mbox{co}(V)=\mbox{co}(V_3)$ holds for $\eta\in[\eta_4,\pi/3)$. For the purpose it suffices to show
\begin{equation}\label{gjw}
\{z_j:j>3\}\cup\{w_j:j>3\}\subset\mbox{co}(V_3).
\end{equation}

\begin{lemma}\label{V31}
For each $\eta\in[\eta_4,\pi/3)$ we have $z_4,z_5,w_4,w_5\in \mbox{co}(V_3)$.
\end{lemma}
\begin{proof}
Let $\eta\in[\eta_4,\pi/3)$. We first show
\begin{equation}\label{z4}
z_4\in \mbox{co}\{0,z_3,w_1,w_2\}.
\end{equation}

On the one hand, we have $0<2\pi-5\eta=\arg z_4<\arg z_3=2\pi-4\eta<\pi$. On the other hand, we have $\sin 5\eta<0$ and $\sin3\eta>0$, and hence
\begin{eqnarray*}
\mbox{Im}\frac{-z_4}{w_2-z_4}&\asymp&\mbox{Im}(-z_4(\overline{w}_2-\overline{z}_4))=\mbox{Im}(-z_4\overline{w}_2)
=\mbox{Im}(-z_4(1-c|a|^2\overline{a}^2)) \\
&\asymp& \mbox{Im}(-a^5+c|a|^6a^3)\asymp\sin5\eta-c|a|^4\sin3\eta<0,
\end{eqnarray*}
which implies $\angle w_2z_40>\pi$. In addition, we have $\Phi_4(\eta)\leq0$ by the proof of Lemma \ref{L1.2}, and hence $\angle z_3z_4w_1\geq\pi$. Finally, we have
\begin{eqnarray*}
\mbox{Im} w_1-\mbox{Im} z_4 &=& c|a|^3\sin\eta+c|a|^5\sin 5\eta \asymp 1+|a|^2\frac{\sin 5\eta}{\sin\eta}\\
&=&1+\frac{1-3|a|^2+|a|^4}{|a|^2}\asymp1-2|a|^2+|a|^4>0,
\end{eqnarray*}
and hence $\mbox{Im} z_4<\mbox{Im} w_1=\mbox{Im} w_2$. These facts imply (\ref{z4}). Thus $z_4\in\mbox{co}(V_3)$.

Since $f_2$ is a similarity map, by applying $f_2$ to the relationship (\ref{z4}), we get
$$w_4\in\mbox{co}\{1,w_3,b_1,z_0\},$$
and hence $w_4\in\mbox{co}(V_3)$.

Similarly, by applying $f_1$ and $f_2\circ f_1$ to (\ref{z4}) respectively, we get
$$z_5\in \mbox{co}\{0,z_4,f_1(w_1),f_1(w_2)\}\, \mbox{ and }\, w_5\in \mbox{co}\{1,w_4,f_2\circ f_1(w_1),f_2\circ f_1(w_2)\}.$$
As was shown, $z_4, w_4\in\mbox{co}(V_3)$. To prove $z_5, w_5\in\mbox{co}(V_3)$, it suffices to show
\begin{equation}\label{ch}
f_1(w_1),\, f_1(w_2),\, f_2\circ f_1(w_1),\, f_2\circ f_1(w_2)\in \mbox{co}(V_3).
\end{equation}

By Lemma \ref{1L2.1}, $w_0\in\mbox{co}\{0,1,z_2,z_3\}$. It then follows that
$$f_1(w_0)\in\mbox{co}\{0,a,z_3,z_4\} \,\mbox{ and }\, f_2(w_0)\in \mbox{co}\{1,a,w_2,w_3\},$$
and hence $f_1(w_0),f_2(w_0)\in \mbox{co}(V_3)$ since $a=c^{-1}z_0\in \mbox{co}(V_3)$. On the other hand, by  Lemma \ref{ylw}, we have
$$
f_1(w_{k+1})=(1-|a|^2)f_1(w_k)+|a|^2f_2(w_k).
$$
By using this formula, we get $f_1(w_1)\in \mbox{co}(V_3)$. Now, since $f_2(w_1)=b_1\in V_3$, we get $f_1(w_2)\in \mbox{co}(V_3)$ by using this formula again. Finally, observing that
$$f_2\circ f_1(w_1)=f_2(a-c|a|^2a^2)=1-\bar{a}(a-c|a|^2a^2)=1-|a|^2+|a|^4z_0$$
and
$$f_2\circ f_1(w_2)=f_2(a-c|a|^2a^3)=1-\bar{a}(a-c|a|^2a^3)=1-|a|^2+|a|^4z_1,$$
we have $f_2\circ f_1(w_1),f_2\circ f_1(w_2)\in \mbox{co}(V_3)$ by using $1, |a|^2z_0,|a|^2z_1\in \mbox{co}(V_3)$. This proves (\ref{ch}), and thus completes the proof.
\end{proof}

\begin{lemma}\label{V32}
For each $\eta\in[\eta_4,\pi/3)$ we have $z_6,w_6\in\mbox{co}(V_3)$.
\end{lemma}
\begin{proof} Let $\eta\in[\eta_4,\pi/3)$ be given. For the purpose it suffices to prove
\begin{equation}\label{wd}
z_6\in \mbox{co}\{0,w_2,z_0\} \mbox{ and }w_6\in \mbox{co}\{1,z_0,w_0\}.
\end{equation}

First, we have $\mbox{Im}w_2=c|a|^4\sin2\eta>0$. On the other hand, by $2|a|\cos\eta=1$, we have $$\mbox{Re}w_2=1-c|a|^4\cos2\eta=1+c|a|^4-\frac{c|a|^2}{2}=c-\frac{c|a|^2}{2}>0.$$ Thus $w_2$ is in the first quadrant.

Secondly, since $\eta_4\in(\pi/4,\pi/3)$ satisfies $8\cos^4\eta_4=1$, by using the known inequality $8\cos^4(2\pi/7)>1$, we get $2\pi/7<\eta_4$, and hence the inequality
\begin{equation}\label{7}
2\pi<7\eta<2\pi+\eta
\end{equation}
holds for the given $\eta$. Moreover, we have
$$\arg z_0=2\pi-\eta<4\pi-7\eta=\arg z_6<2\pi.$$

Finally, we have $\angle z_0z_6w_2\in(0,\pi)$ by
\begin{eqnarray*}
&&\mbox{Im}\frac{w_2-z_6}{z_0-z_6}\asymp\mbox{Im}((\bar{z}_0-\bar{z}_6)(w_2-z_6))=\mbox{Im}((\bar{z}_0-\bar{z}_6)(w_2-z_0))\\
&\asymp&\mbox{Im}((\bar{a}-\bar{a}^7)(1-c|a|^2a^2-ca))=\mbox{Im}(\bar{a}-\bar{a}^7-c|a|^4a+c|a|^6\bar{a}^5+c|a|^2\bar{a}^6)\\
&\asymp& (1+c|a|^4)\sin\eta+c|a|^{10}\sin5\eta+c|a|^7\sin6\eta-|a|^6\sin7\eta\\
&\asymp& c+c|a|^{10}\frac{\sin5\eta}{\sin\eta}+c|a|^7\frac{\sin6\eta}{\sin\eta}-|a|^6\frac{\sin7\eta}{\sin\eta}\\
&\asymp& 1+|a|^6-3|a|^8+|a|^{10}+|a|^2-4|a|^4+3|a|^6-c^{-1}(1-5|a|^2+6|a|^4-|a|^6)\\
&=& 1+|a|^2-4|a|^4+4|a|^6-3|a|^8+|a|^{10}-(1-|a|^4)(1-5|a|^2+6|a|^4-|a|^6)\\
&=& 6|a|^2-9|a|^4+3|a|^8\asymp2-3|a|^2+|a|^6>0,
\end{eqnarray*}
where the positivity is due to that the function $f(x)=2-3x+x^3$ is strictly decreasing on $[0,1]$ and takes $0$ at $x=1$.

The above facts verify the relationship $z_6\in \mbox{co}\{0,w_2,z_0\}$. By applying $f_2$ to it, we get $w_6\in \mbox{co}\{1,z_0,w_0\}$. This proves (\ref{wd}), and thus $z_6,w_6\in\mbox{co}(V_3)$.
\end{proof}

By applying $f_1$ to $z_6\in \mbox{co}\{0,w_2,z_0\}$, we easily get $z_7\in\mbox{co}(V_3)$. We require the following finer result.

\begin{lemma}\label{V33}
If $\eta\in[\eta_4,\pi/3)$ then $z_7\in \mbox{co}\{0,z_0,z_1\}$ and $w_7\in \mbox{co}\{1,w_0,w_1\}$.
\end{lemma}
\begin{proof} Let $\eta\in[\eta_4,\pi/3)$. By (\ref{7}), we have $2\pi+\eta<8\eta<2\pi+2\eta$. Thus
$$\arg z_1<\arg z_7<\arg z_0.$$
On the other hand, we have
\begin{eqnarray*}
\mbox{Im}\frac{z_0-z_7}{z_1-z_7}&\asymp&\mbox{Im}((\bar{z}_1-\bar{z}_7)(z_0-z_7))=\mbox{Im}((\bar{z}_1-\bar{z}_0)(z_0-z_7))\\
&\asymp&\mbox{Im}(z_7-z_0)\asymp\mbox{Im}(a^8-a)\asymp\sin\eta-|a|^7\sin8\eta\\
&\asymp&1-\frac{|a|^7\sin8\eta}{\sin\eta}=1-(1-6|a|^2+10|a|^4-4|a|^6)\\
&\asymp& 3-5|a|^2+2|a|^4=(1-|a|^2)(3-2|a|^2)>0,
\end{eqnarray*}
which implies $\angle z_1z_7z_0\in(0,\pi)$. Thus $z_7\in \mbox{co}\{0,z_0,z_1\}$. By applying $f_2$ to this relationship, we get $w_7\in \mbox{co}\{1,w_0,w_1\}$.
\end{proof}

Let $\eta\in[\eta_4,\pi/3)$. The above lemmas yield $\{z_1,\cdots,z_7, w_1,\cdots w_7\}\subset\mbox{co}(V_3)$. Now, applying $f_1^k$ to $z_7\in \mbox{co}\{0,z_0,z_1\}$ gives $z_{k+7}\in \mbox{co}\{0,z_k,z_{k+1}\}$, and then applying  $f_2$ to $z_{k+7}\in \mbox{co}\{0,z_k,z_{k+1}\}$ gives $w_{k+7}\in \mbox{co}\{1,w_k,w_{k+1}\}$, where $k$ is a positive integer. By using these facts, we can prove (\ref{gjw}) by induction.
It then follows that $\mbox{co}(V)=\mbox{co}(V_3)$.

\section{Disk Properties of Dragon Curves}

Let $\eta\in(0,\pi/3)$. Denote by $D_k$ the closed disk $|z|\leq |z_k|$. Since $|z_k|$ is decreasing, one has $\{z_j: j\geq k\}\subset D_k$ for every integer $k\geq 1$. We have the following result.

\begin{lemma}\label{disc}
Let $\eta\in (0,\eta_4)$. Then, for every integer $k\geq 1$, the broken line $[z_0,z_1,\cdots, z_{k}]$ lies outside the disk $D_{k+1}$.
\end{lemma}
\begin{proof}
Let $\eta\in (0,\eta_4)$ and $k\geq 1$.
To prove Lemma \ref{disc}, it suffices to show
$$|z_{k-1}+t(z_k -z_{k-1})|>|z_{k+1}|$$
for all $t\in[0,1]$, which reduces to proving
$$|1-t\bar{a}|^2>|a|^4$$
for all $t\in[0,1]$. Actually, we have
$$
|1-t\bar{a}|^2-|a|^4=1-t\bar{a}-ta+t^2|a|^2-|a|^4=1-t+t^2|a|^2-|a|^4.
$$
If $2|a|^2\leq1$, we have
$$1-t+t^2|a|^2-|a|^4 \geq 2(1-t)|a|^2+t^2|a|^2-|a|^4 \geq |a|^2-|a|^4>0.$$
If $2|a|^2>1$, we have by (\ref{eta4})
\begin{eqnarray*}
1-t+t^2|a|^2-|a|^4 &=& 1-|a|^4+(t|a|-\frac{1}{2|a|})^2-\frac{1}{4|a|^2}\\
&\geq& 1-|a|^4-\frac{1}{4|a|^2}\asymp 4|a|^2-4|a|^6-1\\&>& 2|a|^2-4|a|^6=2|a|^2(1-2|a|^4)>0.
\end{eqnarray*}
This completes the proof. \end{proof}

\begin{lemma}\label{disc1}
Let $k\geq 4$ be an integer. Then the following propositions hold.

1) If $\eta\in [\eta_{k+1},\pi/k)$ then the segment $[z_0,z_{2k+1}]$ lies outside $D_{2k+2}$.

2) If $\eta\in [\pi/k,\eta_k)$ then the segment $[z_0,z_{2k-1}]$ lies outside $D_{2k}$.
\end{lemma}

\begin{proof}
1) Let $k\geq 4$ be an integer and $\eta\in [\eta_{k+1},\pi/k)$. Then
$$2\pi-\eta <(2k+1)\eta <2\pi+\eta,$$
and hence
\begin{equation}\label{k1}
\cos(2k+1)\eta>\cos\eta.
\end{equation}

To show that $[z_{2k+1},z_0]$ lies outside the disk $D_{2k+2}$, it suffices to show
$$
\min_{t\in[0,1]}|(1-t)z_0+tz_{2k+1}|=|z_{2k+1}|,
$$
which reduces to proving
\begin{equation}\label{eq1}
\min_{t\in[0,1]}|1-t+ta^{2k+1}|=|a|^{2k+1}.
\end{equation}

Let $\varphi(t)=|1-t+ta^{2k+1}|^2$. By using (\ref{k1}) and $2|a|\cos\eta=1$, we have
\begin{eqnarray*}
\varphi(t)&=&(1-t)^2+2t(1-t)\mbox{Re}a^{2k+1}+t^2|a|^{4k+2}\\
&=& (1-t)^2+2t(1-t)|a|^{2k+1}\cos(2k+1)\eta+t^2|a|^{4k+2}\\
&\geq& (1-t)^2+2t(1-t)|a|^{2k+1}\cos\eta+t^2|a|^{4k+2}\\
&=& (1-t)^2+t(1-t)|a|^{2k}+t^2|a|^{4k+2}.
\end{eqnarray*}
For $t\in(0,1)$ we have by (\ref{eta4})
\begin{eqnarray*}
\varphi'(t) &=& -2(1-t)+(1-2t)|a|^{2k}+2t|a|^{4k+2} \\
&=& -2+|a|^{2k}+2t(1-|a|^{2k}+|a|^{4k+2})\\
&<& -|a|^{2k}+2|a|^{4k+2}=-|a|^{2k}(1-2|a|^{2k+2})<0,
\end{eqnarray*}
which implies that the function $\varphi(t)$ is decreasing on $[0,1]$, and hence
$$\min_{t\in[0,1]}\varphi(t)=\varphi(1)=|a|^{4k+2}.$$
This proves (\ref{eq1}).

\medskip

2) Let $k\geq 4$ be an integer and $\eta\in [\pi/k,\eta_k)$. We have
$$2\pi-\eta\leq(2k-1)\eta <2\pi+\eta,$$
so
$$
\cos(2k-1)\eta\geq\cos\eta.
$$

To show that $[z_{2k-1},z_0]$ lies outside $D_{2k}$, it suffices to show
$$
\min_{t\in[0,1]}|(1-t)z_0+tz_{2k-1}|=|z_{2k-1}|,
$$
which reduces to proving
$$
\min_{t\in[0,1]}|1-t+ta^{2k-1}|=|a|^{2k-1}.
$$
This proof is analogous to that of (\ref{eq1}) and is left to the readers.
\end{proof}

\begin{lemma}\label{64}
Let $\eta\in(0,\eta_4)$. The following propositions hold.

1) $|z_j|<\mbox{Re} z_0<1$ for all integers $j\geq 3$.

2) $|z_j|<\min_{t\in[0,1]}|(1-t)z_0+tz_1|$ for all integers $j\geq 2$.
\end{lemma}
\begin{proof}
1) Let $\eta\in(0,\eta_4)$ and $j\geq 3$. Arguing by using (\ref{eta4}), we get $$c<2\,\mbox{ and }\, |z_j|=c|a|^{j+1}<{c}/{2}=c|a|\cos\eta=\mbox{Re} z_0<1.$$

2) Let $\eta\in(0,\eta_4)$ and $j\geq 2$. To prove the desired inequality, it suffices to show
$$|a|^{2j}<\min_{t\in[0,1]}\varphi(t),$$
where $\varphi(t)=|1-t+ta|^2$.
In fact, we have
\begin{eqnarray*}
\varphi(t) &=& (1-t)^2+2t(1-t)\mbox{Re} a+t^2|a|^2 \\
&=& (1-t)^2+t(1-t)+t^2|a|^2\\
&=&1-t+t^2|a|^2
\end{eqnarray*}
and $$\varphi'(t)=-1+2t|a|^2.$$

If $2|a|^2<1$ then $\varphi(t)$ is decreasing on $[0,1]$, so $\varphi(t)\geq \varphi(1)=|a|^2>|a|^{2j}$. In the other case where $2|a|^2\geq1$, we get by using  (\ref{eta4})
$$\varphi(t)=1-t+t^2|a|^2\geq 1-\frac{1}{4|a|^2}\geq\frac{1}{2}>|a|^{2j}.$$
This completes the proof.
\end{proof}

\section{The Proof of Theorem \ref{tm2}}

The proof of Theorem \ref{tm1} establishes $\mbox{co}(K_\eta)=\mbox{co}(V)$ for every $\eta\in(0,\pi/3)$. For every integer $k\geq 3$ and every $\eta\in[\eta_{k+1},\eta_k)$, Lemma \ref{q3} asserts that $\mbox{co}(V_k)$ is the polygon of vertices $b_1,z_0,z_1,\cdots, z_k,w_1,w_2,\cdots, w_k$ in clockwise order. To complete the proof of Theorem \ref{tm2}, the remaining work is to verify
\begin{equation}\label{jdl}
\mbox{co}(V)=\mbox{co}(V_k)
\end{equation}
for the above given $k$ and $\eta$. This work has been done for $k=3$ and $\eta\in[\eta_4,\pi/3)$ in Section 5. In this section, we prove (\ref{jdl})
for $k\geq 4$ and $\eta\in[\eta_{k+1},\eta_k)$.

Let $k\geq 4$ be an integer and $\eta\in[\eta_{k+1},\eta_k)$. To prove (\ref{jdl}), it suffices to show
\begin{equation}\label{reduc}
\{z_j: j>k\}\cup\{w_j: j>k\}\subset\mbox{co}(V_k).
\end{equation}
In the case $\eta\in[\eta_{k+1},\pi/k)$, by using Lemmas \ref{disc} and \ref{disc1}, we can prove
$$\{z_j: j\geq 2k+2\}\subset D_{2k+2}\subset\mbox{co}\{z_0,z_1,\cdots, z_{2k+1}\}.$$
Then applying $f_2$ to this relationship yields
$$\{w_j: j\geq2k+2\}\subset \mbox{co}\{w_0,w_1,\cdots, w_{2k+1}\}.$$
In the case $\eta\in[\pi/k,\eta_k)$, by using Lemmas \ref{disc} and \ref{disc1}, we can prove
$$\{z_j: j\geq2k\}\subset D_{2k}\subset\mbox{co}\{z_0,z_1,\cdots, z_{2k-1}\}.$$
Then applying $f_2$ to this relationship yields
$$\{w_j: j\geq2k\}\subset \mbox{co}\{w_0,w_1,\cdots, w_{2k-1}\}.$$
Therefore, to prove (\ref{reduc}), it suffices to prove the following lemma.

\begin{lemma}\label{p21} Let $k\geq 4$. The following propositions hold.

1) If $\eta\in[\eta_{k+1},{\pi}/k)$ then
$$
z_{k+1},z_{k+2},\cdots, z_{2k+1}, w_{k+1},w_{k+2},\cdots, w_{2k+1}\in \mbox{co}(V_k).
$$

2) If $\eta\in[{\pi}/k,\eta_k)$ then
$$
z_{k+1},z_{k+2},\cdots, z_{2k-1}, w_{k+1},w_{k+2},\cdots, w_{2k-1}\in \mbox{co}(V_k).
$$
\end{lemma}

The proof of Lemma \ref{p21} is based on the following lemmas.

\begin{lemma}\label{theta}
If $k\geq 4$ is an integer and $\eta\in[\eta_{k+1},\eta_k)$ then $\angle1z_jw_1\in(0,\pi)$ for every integer $j\geq k$.
\end{lemma}
\begin{proof} Let $k\geq 4$, $\eta\in[\eta_{k+1},\eta_k)$, and $j\geq k$ be given. We have
\begin{eqnarray*}
\mbox{Im}\frac{w_1-z_j}{1-z_j}&\asymp&\mbox{Im}((1-\bar{z}_j)(w_1-z_j))=\mbox{Im}((1-\bar{z}_j)(w_1-1))\\
&\asymp & \mbox{Im}((c\bar{a}^{j+1}-1)a)\asymp \sin\eta+c|a|^{j+1}\sin j\eta.
\end{eqnarray*}

Let $\Theta_j(\eta)=\sin\eta+c|a|^{j+1}\sin j\eta$. We are going to show  $\Theta_j(\eta)>0$.

If $k=4$, $\eta\in[\eta_5,\eta_4)$, and $j=4$, one has  by using (\ref{ct}) and (\ref{eta4})
\begin{eqnarray*}
\Theta_4(\eta) &\asymp&1+4c|a|^5\cos\eta\cos2\eta=1+2c|a|^4\cos2\eta\\
&\geq&c|a|^4+2c|a|^4\cos2\eta\asymp 1+2\cos2\eta\\ &=& 4\cos^2\eta-1\asymp1-|a|^2>0.
\end{eqnarray*}

If $k=4$, $\eta\in [\eta_5,\pi/4)$, and $j\geq 5$,  one has $2|a|^2<1$, by which we get
\begin{eqnarray*}
\Theta_j(\eta) &\asymp& (1-|a|^4)\sin\eta+|a|^{j+1}\sin j\eta \geq(1-|a|^4)\sin\frac{\pi}{5}-|a|^6 \\
&\geq& (1-|a|^4)\frac{2}{5}-|a|^6\asymp2-2|a|^4-5|a|^6>2-\frac{1}{2}-\frac{5}{8}>0.
\end{eqnarray*}

If $k=4$, $\eta\in [\pi/4, \eta_4)$, and $j\geq 5$, one has  by using (\ref{eta4})
$$
\Theta_j(\eta) \geq\sin\eta-c|a|^6\geq\sin\frac{\pi}{4}-|a|^2\asymp1-\sqrt{2}|a|^2>0.
$$

If $k\geq 5$, $\eta\in[\eta_{k+1},\eta_k)$, and $j\geq k$, one has  $2|a|^2<1$, by which we get
\begin{eqnarray*}
\Theta_j(\eta) &\asymp& (1-|a|^4)\sin\eta+|a|^{j+1}\sin j\eta \\
&>& (1-|a|^4)\frac{2\eta}{\pi}-|a|^{k+1} \\
&>& \frac{3}{4}\frac{2}{k+1}-(\frac{1}{\sqrt{2}})^{k+1}>0.
\end{eqnarray*}

This proves $\Theta_j(\eta)>0$, so $\angle1z_jw_1\in(0,\pi)$ for the given $k$, $\eta$, and $j$.
\end{proof}

\begin{lemma}\label{jiao}
Let $k\geq 4$. The following propositions hold.

1) If $\eta\in[\eta_{k+1},{\pi}/k)$ then
$$\angle z_{j-1}z_{j}w_1\in(\pi,2\pi)\, \mbox{ for }\, k< j< 2k.$$

2) If $\eta\in[{\pi}/k,\eta_k)$ then $$\angle z_{j-1}z_{j}w_1\in(\pi,2\pi)\, \mbox{ for }\,k<j<2k-2.$$
\end{lemma}
\begin{proof}
By Lemma \ref{Im},
$$\mbox{Im}\frac{w_1-z_j}{z_{j-1}-z_j}\asymp\Phi_j(\eta).$$

1) Let $k\geq 4$ be an integer and $\eta\in[\eta_{k+1},{\pi}/k)$. Then $\Phi_{k+1}(\eta)<0$ by the proof of Lemma \ref{L1.2}. On the other hand, one has $1-2|a|^2>0$ and
\begin{equation}\label{ty}
\pi<(k+1)\eta<\pi+\eta<(k+2)\eta<\cdots<(2k-1)\eta<2\pi-\eta.
\end{equation}
Thus, for $j\in\{k+2,k+3,\cdots,2k-1\}$ one has
$$
\sin (j-1)\eta\leq 0\, \mbox{ and } \sin j\eta<-\sin\eta,
$$
which gives
\begin{eqnarray*}
\Phi_j(\eta)&=& (1-|a|^2-|a|^4)\sin (j-1)\eta+|a|^3\sin j\eta+|a|^j\sin\eta\\
&<& -|a|^3\sin\eta+|a|^j\sin\eta<0.
\end{eqnarray*}
This proves $\Phi_j(\eta)<0$ for $k<j<2k$, so $\angle z_{j-1}z_{j}w_1\in(\pi,2\pi)$ for such $j$.

2) Let $k\geq 4$ be an integer and $\eta\in[{\pi}/k, \eta_k)$. Then
\begin{equation}\label{dh}
\pi\leq k\eta<\pi+\eta\leq(k+1)\eta<(k+2)\eta<\cdots<(2k-3)\eta<2\pi-\eta.
\end{equation}
Thus, for $j\in\{k+1,k+2,\cdots,2k-3\}$, we have
\begin{equation}\label{xy1}
\sin (j-1)\eta\leq 0\, \mbox{ and }\, \sin j\eta<-\sin\eta.
\end{equation}
In the case $k=4$, the given $j$ is equal to $5$. By using (\ref{eta4}) and (\ref{xy1}), we get
\begin{eqnarray*}
\Phi_5(\eta)&=& (1-|a|^2-|a|^4)\sin 4\eta+|a|^3\sin 5\eta+|a|^5\sin\eta\\
&<&(|a|^4 -|a|^2)\sin 4\eta -|a|^3\sin\eta+|a|^5\sin\eta\\
&\asymp&-\sin 4\eta-|a|\sin\eta\asymp-4\cos\eta\cos2\eta-|a|\\
&\asymp& -2\cos 2\eta-|a|^2=2-\frac{1}{|a|^2}-|a|^2<0.
\end{eqnarray*}
In the case $k\geq 5$, one has $1-2|a|^2>0$, which, together with (\ref{xy1}), implies
$$
\Phi_j(\eta)<-|a|^3\sin\eta+|a|^j\sin\eta<0.
$$
This proves $\Phi_j(\eta)<0$ for $k<j<2k-2$, and hence $\angle z_{j-1}z_{j}w_1\in(\pi,2\pi)$. \end{proof}

\begin{lemma}\label{posi00} Let $k\geq 4$. The following propositions hold.

1) If $\eta\in[\eta_{k+1},{\pi}/k)$ then
$z_j\in \mbox{co}\{0, 1, w_1, z_k\}$ for $k<j<2k$.

2) If $\eta\in[{\pi}/k,\eta_k)$ then
$z_j\in \mbox{co}\{0, 1, w_1, z_k\}$ for $k<j<2k-2$.
\end{lemma}
\begin{proof}
1) Let $k\geq 4$ be an integer and $\eta\in[\eta_{k+1},{\pi}/k)$. By using (\ref{ty}), we get
$$\pi>\arg z_k>\arg z_{k+1}>\cdots>\arg z_{2k-1}>0.$$
On the other hand, by Lemmas \ref{theta} and \ref{jiao}, we have $\angle 1z_jw_1\in(0,\pi)$ and $\angle z_{j-1}z_{j}w_1\in(\pi,2\pi)$.
These facts imply $z_j\in \mbox{co}\{0, 1, w_1, z_{j-1}\}$ for $k<j<2k$, which allows us to get $z_j\in \mbox{co}\{0, 1, w_1, z_k\}$ by induction.

2) Let $k\geq 4$ be an integer and $\eta\in[{\pi}/k, \eta_k)$. By using (\ref{dh}), we get
$$\pi>\arg z_k>\arg z_{k+1}>\cdots>\arg z_{2k-3}>0.$$
On the other hand, by Lemmas \ref{theta} and \ref{jiao}, we have $\angle 1z_jw_1\in(0,\pi)$ and $\angle z_{j-1}z_{j}w_1\in(\pi,2\pi)$ for $k<j<2k-2$.  These facts imply $z_j\in \mbox{co}\{0, 1, w_1, z_{j-1}\}$, which allows us to get $z_j\in \mbox{co}\{0, 1, w_1, z_k\}$ for $k<j<2k-2$ by induction.
\end{proof}

\begin{lemma}\label{54}
Let $k\geq 4$. The following propositions hold.

1) If $\eta\in[\eta_{k+1},\pi/k)$ then $z_{2k}, z_{2k+1}\in \mbox{co}\{0, 1,z_0, z_1, z_{2k-1}\}$.

2) If $\eta\in[\pi/k,\eta_k)$ then $z_{2k-2}, z_{2k-1}\in \mbox{co}\{0, 1,z_0, z_1, z_{2k-3}\}$.
\end{lemma}
\begin{proof}
1) Let $k\geq 4$ be an integer and $\eta\in[\eta_{k+1},\pi/k)$. Then we have
\begin{equation}\label{lh}
2k\eta\in(2\pi-2\eta,2\pi)\subset (\pi+2\eta,2\pi)\,\mbox{ and }\, \arg z_{2k-1}=2\pi-2k\eta\in(0,2\eta).
\end{equation}

Case 1. $\arg z_{2k-1}\in(0,\eta]$.

\medskip

In this case, it holds either $\arg z_0<\arg z_{2k}<2\pi$ or $\arg z_{2k}=0$. On the other hand, we have
$|z_{2k}|<\mbox{Re} z_0<1$ by Lemma \ref{64}. Thus, $z_{2k}\in \mbox{co}\{0, z_0,1\}$.

As for $z_{2k+1}$, one has $\arg z_1<\arg z_{2k+1}\leq\arg z_0$. In addition, one has by Lemma \ref{64} $$|z_{2k+1}|<\min_{t\in[0,1]}|(1-t)z_0+tz_1|.$$ Thus $z_{2k+1}\in \mbox{co}\{0, z_1, z_0\}$.

\medskip

Case 2. $\arg z_{2k-1}\in(\eta,2\eta)$.

\medskip

In this case, one has $0<\arg z_{2k}<\arg z_{2k-1}$.
We claim $z_{2k}\in \mbox{co}\{0, 1, z_{2k-1}\}$. To prove this claim, it suffices to show that
$\angle 1z_{2k}z_{2k-1}\in(0,\pi)$, which is reduced to showing
\begin{equation}\label{2t+0}
\mbox{Im}((1-\bar{z}_{2k})(z_{2k-1}-z_{2k}))>0.
\end{equation}

By the first relationship of (\ref{lh}), one has $$(2k-1)\eta\in(\pi+\eta,2\pi-\eta),$$
and hence $\sin(2k-1)\eta<-\sin\eta$. It then follows by (\ref{eta4}) that
\begin{eqnarray*}
&&\mbox{Im}((1-\bar{z}_{2k})(z_{2k-1}-z_{2k}))
\asymp\mbox{Im}(1-c\bar{a}^{2k+1})a^{2k-1}\\ &=&\mbox{Im}(a^{2k-1}-c|a|^{4k-2}\bar{a}^2)\asymp-\sin(2k-1)\eta-c|a|^{2k+1}\sin2\eta\\
&=&-\sin(2k-1)\eta-c|a|^{2k}\sin\eta>(1-c|a|^{2k})\sin\eta>0.
\end{eqnarray*}
This proves the inequality (\ref{2t+0}), so we have $z_{2k}\in \mbox{co}\{0, 1, z_{2k-1}\}$.

As for $z_{2k+1}$ in Case 2, one has $\arg z_0<\arg z_{2k+1}<2\pi$. In addition, one has
$|z_{2k+1}|<\mbox{Re} z_0<1$ by Lemma \ref{64}. Thus $z_{2k+1}\in \mbox{co}\{0, z_0,1\}$.

\medskip

2) Let $k\geq 4$ be an integer and $\eta\in[\pi/k,\eta_k)$. Then we have
$$(2k-2)\eta\in[2\pi-2\eta, 2\pi)\,\mbox{ and }\,\arg z_{2k-3}=2\pi-(2k-2)\eta\in(0,2\eta].$$

Case 3. $\arg z_{2k-3}\in(0,\eta]$.

\medskip

In this case, it holds either $\arg z_0<\arg z_{2k-2}<2\pi$ or $\arg z_{2k}=0$. In addition, we have
$|z_{2k-2}|<\mbox{Re} z_0<1$ by Lemma \ref{64}. Thus, $z_{2k-2}\in \mbox{co}\{0, z_0,1\}$.

As for the point $z_{2k-1}$, one has $\arg z_1<\arg z_{2k-1}\leq\arg z_0$. In addition, one has by Lemma \ref{64} $$|z_{2k-1}|<\min_{t\in[0,1]}|(1-t)z_0+tz_1|.$$ Thus $z_{2k-1}\in \mbox{co}\{0, z_1, z_0\}$.

\medskip

Case 4. $\arg z_{2k-3}\in(\eta,2\eta]$.

\medskip

In this case, $0<\arg z_{2k-2}<\arg z_{2k-3}$.
We claim $z_{2k-2}\in \mbox{co}\{0, 1, z_{2k-3}\}$. To prove this claim, it suffices to show that
$\angle 1z_{2k-2}z_{2k-3}\in(0,\pi)$, which is reduced to showing
\begin{equation}\label{2t+}
\mbox{Im}((1-\bar{z}_{2k-2})(z_{2k-3}-z_{2k-2}))>0.
\end{equation}

If $k>4$, one has $4\eta<(k-1)\eta<\pi$ by the assumption of 2), and hence
$$(2k-2)\eta\in[2\pi-2\eta, 2\pi)\subset(\pi+2\eta, 2\pi),$$
which implies
$\sin(2k-3)\eta<-\sin\eta$. Thus
\begin{eqnarray*}
\mbox{Im}((1-\bar{z}_{2k-2})(z_{2k-3}-z_{2k-2}))
&\asymp& -\sin(2k-3)\eta-c|a|^{2k-2}\sin\eta\\&>&(1-c|a|^{2k-2})\sin\eta>0.
\end{eqnarray*}

If $k=4$, the assumption of 2) and the inequality (\ref{eta4}) yield $$2|a|^2\geq1> 2|a|^4.$$ Therefore
\begin{eqnarray*}
&&\mbox{Im}((1-\bar{z}_{6})(z_{5}-z_{6}))\asymp -\sin5\eta-c|a|^6\sin\eta\asymp-\frac{\sin5\eta}{\sin\eta}-c|a|^6\\
&\asymp&-\frac{1-3|a|^2+|a|^4}{|a|^4}-c|a|^6\asymp(1-|a|^4)(3|a|^2-1-|a|^4)-|a|^{10}\\
&\asymp&3|a|^2-1-3|a|^6+|a|^8-|a|^{10}>0,
\end{eqnarray*}
where the positivity follows by the convexity of the function
$$f(x)=3x-1-3x^3+x^4-x^5$$
on the interval $[{1}/{2},{1}/{\sqrt{2}})$. Indeed, $f''(x)=-18x+12x^2-20x^3$. By a simple computation, one has  $f''(x)<0$ on $[{1}/{2},{1}/{\sqrt{2}})$, and hence $f$ is convex satisfying $$f(x)\geq\min\{f({1}/{2}),f({1}/{\sqrt{2}})\}>0$$
on $[{1}/{2},{1}/{\sqrt{2}})$. This proves (\ref{2t+}), so we have $z_{2k-2}\in \mbox{co}\{0, 1, z_{2k-3}\}$.

As for the point $z_{2k-1}$ in Case 4, it holds either $\arg z_0<\arg z_{2k-1}<2\pi$ or $\arg z_{2k-1}=0$.
In addition, $|z_{2k-1}|<\mbox{Re} z_0<1$ by Lemma \ref{64}. These two facts imply $z_{2k-1}\in \mbox{co}\{0, z_0,1\}$.
This completes the proof of the lemma.
\end{proof}

{\bf The Proof of Lemma \ref{p21}.} Let $k\geq 4$ be an integer and  $\eta\in[\eta_{k+1},\pi/k)$. By Lemma \ref{posi00} we have
$$z_{k+1},z_{k+2},\cdots, z_{2k-1}\in \mbox{co}\{0, 1, w_1, z_k\}.$$
By applying $f_2$ to this relationship, we get
$$w_{k+1},w_{k+2},\cdots, w_{2k-1}\in \mbox{co}\{1, a, b_1, w_k\}.$$
Thus $$z_{k+1},z_{k+2},\cdots, z_{2k-1}, w_{k+1},w_{k+2},\cdots, w_{2k-1}\in \mbox{co}(V_k).$$
Furthermore, by Lemma \ref{54} we have
$$z_{2k}, z_{2k+1}\in \mbox{co}\{0, 1,z_0, z_1, z_{2k-1}\}.$$
Applying $f_2$ to this relationship yields
$$w_{2k}, w_{2k+1}\in \mbox{co}\{1, a, w_0, w_1, w_{2k-1}\}.$$
It then follows that $$z_{2k}, z_{2k+1},w_{2k}, w_{2k+1}\in \mbox{co}(V_k).$$
To sum up, we have
$$z_{k+1},z_{k+2},\cdots, z_{2k+1}, w_{k+1},w_{k+2},\cdots, w_{2k+1}\in \mbox{co}(V_k).$$

Let $k\geq 4$ be an integer and $\eta\in[\pi/k, \eta_{k})$. Arguing by using Lemmas \ref{posi00} and \ref{54} as we just did, we get
$$z_{k+1},z_{k+2},\cdots, z_{2k-1}, w_{k+1},w_{k+2},\cdots, w_{2k-1}\in \mbox{co}(V_k).$$
This proves Lemma \ref{p21}, and thus completes the proof of Theorem \ref{tm2}.

\end{document}